\documentclass[12pt]{amsart}
\usepackage{SSdefn}
\DeclareMathOperator{\qrank}{qrk}

\DeclareMathOperator{\maxrank}{maxrank}
\DeclareMathOperator{\minrank}{minrank}
\DeclareMathOperator{\Surj}{Surj}

\title{Topological noetherianity for cubic polynomials}

\author{Harm Derksen}
\address{Department of Mathematics, University of Michigan, Ann Arbor, MI}
\email{\href{mailto:hderksen@umich.edu}{hderksen@umich.edu}}
\urladdr{\url{http://www.math.lsa.umich.edu/~hderksen/}}

\author{Rob H.\ Eggermont}
\address{Department of Mathematics, University of Michigan, Ann Arbor, MI}
\email{\href{mailto:robegger@umich.edu}{robegger@umich.edu}}
\urladdr{\url{http://www-personal.umich.edu/~robegger/}}

\author{Andrew Snowden}
\address{Department of Mathematics, University of Michigan, Ann Arbor, MI}
\email{\href{mailto:asnowden@umich.edu}{asnowden@umich.edu}}
\urladdr{\url{http://www-personal.umich.edu/~asnowden/}}

\thanks{HD was supported by NSF grant DMS-1601229. AS was supported by NSF grants DMS-1303082 and DMS-1453893.}

\date{\today}

\begin{document}

\begin{abstract}
Let $P_3(\bk^{\infty})$ be the space of complex cubic polynomials in infinitely many variables over the field $\bk$. We show that this space is $\GL_{\infty}$-noetherian, meaning that any $\GL_{\infty}$-stable Zariski closed subset is cut out by finitely many orbits of equations. Our method relies on a careful analysis of an invariant of cubics we introduce called q-rank. This result is motivated by recent work in representation stability, especially the theory of twisted commutative algebras. It is also connected to uniformity problems in commutative algebra in the vein of Stillman's conjecture.
\end{abstract}

\maketitle
\tableofcontents

\section{Introduction}

Let $P_d(\bk^n)$ be the space of degree $d$ polynomials in $n$ variables over an algebraically closed field $\bk$ of characteristic $\ne 2,3$. Let $P_d(\bk^{\infty})$ be the inverse limit of the $P_d(\bk^n)$, equipped with the Zariski topology and its natural $\GL_{\infty}$ action (see \S \ref{ss:note}). This paper is concerned with the following question:

\begin{question} \label{mainques}
Is the space $P_d(\bk^{\infty})$ noetherian with respect to the $\GL_{\infty}$ action? That is, can every Zariski-closed $\GL_{\infty}$-stable subspace be defined by finitely many orbits of equations?
\end{question}

This question may seem somewhat esoteric, but it is motivated by recent work in the field of representation stability, in particular the theory of twisted commutative algebras; see \S \ref{ss:tca}. It is also connected to certain uniformity questions in commutative algebra in the spirit of (the now resolved) Stillman's conjecture; see \S \ref{ss:stillman}.

For $d \le 2$ the question is easy since one can explicitly determine the $\GL_{\infty}$ orbits on $P_d(\bk^{\infty})$. For $d \ge 3$ this is not possible, and the problem is much harder. The purpose of this paper is to settle the $d=3$ case:

\begin{theorem} \label{mainthm}
Question~\ref{mainques} has an affirmative answer for $d=3$.
\end{theorem}

In fact, we prove a quantitative result in finitely many variables that implies the theorem in the limit. This may be of independent interest; see \S \ref{ss:overview} for details.

\subsection{Overview of proof} \label{ss:overview}

The key concept in the proof, and the focus of most of this paper, is the following notion of rank for cubic forms.

\begin{definition}
Let $f \in P_3(\bk^n)$ with $n \le \infty$. We define the {\bf q-rank}\footnote{The q here is meant to indicate the presence of quadrics in the expression for $f$.} of $f$, denoted $\qrank(f)$, to be the minimal non-negative integer $r$ for which there is an expression $f = \sum_{i=1}^r \ell_i q_i$ with $\ell_i \in P_1(\bk^n)$ and $q_i \in P_2(\bk^n)$, or $\infty$ if no such $r$ exists (which can only happen if $n=\infty$).
\end{definition}

\begin{example} \label{ex1}
For $n \le \infty$, the cubic
\begin{displaymath}
x_1y_1z_1 + x_2y_2z_2 + \cdots + x_n y_n z_n = \sum_{i=1}^n x_i y_i z_i
\end{displaymath}
has q-rank $n$. This is proved in \S \ref{sec:example}. In particular, infinite q-rank is possible when $n=\infty$.
\end{example}

\begin{example}
The cubic $x^3+y^3$ has q-rank~1, as follows from the identity
\begin{displaymath}
x^3+y^3 = (x+y)(x^2-xy+y^2).
\end{displaymath}
The cubic $\sum_{i=1}^{2n} x_i^3$ therefore has q-rank at most $n$, and we expect it is exactly $n$.
\end{example}

\begin{remark}
The notion of q-rank is similar to some other invariants in the literature:
\begin{enumerate}
\item Ananyan--Hochster \cite{stillman} define a homogeneous polynomial to have {\bf strength} $\ge k$ if it does not belong to an ideal generated by $k$ forms of strictly lower degree. For cubics, q-rank is equal to strength plus one.
\item The paper \cite{slicerank} (inspired by Tao's blog post \cite{tao}) introduced the notion of ``slice rank'' for tensors. Q-rank is basically a symmetric version of this. \qedhere
\end{enumerate}
\end{remark}

Let $P_3(\bk^{\infty})_{\le r}$ be the locus of forms $f$ with $\qrank(f) \le r$. This is the image of the map
\begin{displaymath}
P_2(\bk^{\infty})^r \times P_1(\bk^{\infty})^r \to P_3(\bk^{\infty}), \qquad (q_1, \ldots, q_r, \ell_1, \ldots, \ell_r) \mapsto \sum_{i=1}^r \ell_i q_i.
\end{displaymath}
The main theorem of \cite{eggermont} implies that the domain of the above map is $\GL_{\infty}$-noetherian, and so, by standard facts (see \cite[\S 3]{draisma}), its image $P_3(\bk^{\infty})_{\le r}$ is as well. It follows that any $\GL_{\infty}$-stable closed subset of $P_3(\bk^{\infty})$ of bounded q-rank is cut out by finitely many orbits of equations. Theorem~\ref{mainthm} then follows from the following result:

\begin{theorem} \label{mainthm2}
Any $\GL_{\infty}$-stable subset of $P_3(\bk^{\infty})$ containing forms of arbitrarily high q-rank is Zariski dense.
\end{theorem}

To prove this theorem, one must show that if $f_1, f_2, \ldots$ is a sequence in $P_3(\bk^{\infty})$ of unbounded q-rank then for any $d$ there is a $k$ such that the orbit-closure of $f_k$ projects surjectively onto $P_3(\bk^d)$. We prove a quantitative version of this statement:

\begin{theorem} \label{mainthm3}
Let $f \in P_3(\bk^n)$ have q-rank $r \gg 0$ (in fact, $r>\exp(240)$ suffices), and suppose $d \le \tfrac{1}{3} \log(r)$. Then the orbit closure of $f$ surjects onto $P_3(\bk^d)$.
\end{theorem}

The proof of this theorem is really the heart of the paper. The idea is as follows. Suppose that $f=\sum_{i=1}^m \ell_i q_i$ has large q-rank. We establish two key facts. First, after possibly degenerating $f$ (i.e., passing to a form in the orbit-closure) one can assume that the $\ell_i$'s and $q_i$'s are in separate sets of variables, while maintaining the assumption that $f$ has large q-rank. This is useful when studying the orbit closoure, as it allows us to move the $\ell$'s and $q$'s independently. Second, we show that $q$'s have large rank in a very stong sense: namely, that within the linear span of the $q$'s there is a large-dimensional subspace such that every non-zero element of it has large rank. The results of \cite{eggermont} then imply that the orbit closure of $(q_1, \ldots, q_m; \ell_1, \ldots, \ell_m)$ in $P_2(\bk^n)^m \times P_1(\bk^n)^m$ surjects onto $P_2(\bk^d)^m \times P_2(\bk^d)^m$, and this yields the theorem.

\subsection{Uniformity in commutative algebra} \label{ss:stillman}

We now explain one source of motivation for Question~\ref{mainques}. An {\bf ideal invariant} is a rule that assigns to each homogeneous ideal $I$ in each standard-graded polynomial $\bk$-algebra $A$ (in finitely many variables) a quantity $\nu_A(I) \in \bZ \cup \{\infty\}$, such that $\nu_A(I)$ only depends on the pair $(A,I)$ up to isomorphism. We say that $\nu$ is {\bf cone-stable} if $\nu_{A[x]}(I[x])=\nu_A(I)$, that is, adjoining a new variable does not affect $\nu$. The main theorem of \cite{ermansam} is (in part):

\begin{theorem}[\cite{ermansam}] \label{thm:ess}
The following are equivalent:
\begin{enumerate}
\item Let $\nu$ be a cone-stable ideal invariant that is upper semi-continuous in flat families, and let $\bd=(d_1, \ldots, d_r)$ be a tuple of non-negative integers. Then there exists an integer $B$ such that $\nu_A(I)$ is either infinite or at most $B$ whenever $I$ is an ideal generated by $r$ elements of degrees $d_1, \ldots, d_r$. (Crucially, $B$ does not depend on $A$.)
\item For every $\bd$ as above, the space
\begin{displaymath}
P_{d_1}(\bk^{\infty}) \times \cdots \times P_{d_r}(\bk^{\infty})
\end{displaymath}
is $\GL$-noetherian.
\end{enumerate}
\end{theorem}

\begin{remark}
Define an ideal invariant $\nu$ by taking $\nu_A(I)$ to be the projective dimension of $I$ as an $A$-module. This is cone-stable and upper semi-continuous in flat families. The boundedness in Theorem~\ref{thm:ess}(a) for this $\nu$ is exactly Stillman's conjecture, proved in \cite{stillman}.
\end{remark}

Theorem~\ref{thm:ess} shows that Question~\ref{mainques} is intimately connected to uniformity questions in commutative algbera in the style of Stillman's conjecture. The results of \cite{ermansam} are actually more precise: if part (b) holds for a single $\bd$ then part (a) holds for the corresponding $\bd$. Thus, combined with Theorem~\ref{mainthm}, we obtain:

\begin{theorem} \label{thm:cubicinv}
Let $\nu$ be a cone-stable ideal invariant that is upper semi-continuous in flat families. Then there exists an integer $B$ such that $\nu(I)$ is either infinite or at most $B$, whenever $I$ is generated by a single cubic form.
\end{theorem}

The following two consequences of Theorem~\ref{thm:cubicinv} are taken from \cite{ermansam}.

\begin{corollary}
Given a positive integer $c$ there is an integer $B$ such that the following holds: if $Y \subset \bP^{n-1}$ is a cubic hypersurface containing finitely many codimension $c$ linear subspaces then it contains at most $B$ such subspaces.
\end{corollary}

\begin{corollary}
Given a positive integer $c$ there is an integer $B$ such that the following holds: if $Y \subset \bP^{n-1}$ is a cubic hypersurface whose singular locus has codimension $c$ then its singular locus has degree at most $B$.
\end{corollary}

It would be interesting if these results could be proved by means of classical algebraic geometry. It would also be interesting to determine the bound $B$ for some small values of $c$.

\subsection{Twisted commutative algebras} \label{ss:tca}

In this section we put $\bk=\bC$. Our original motivation for considering Question~\ref{mainques} came from the theory of twisted commutative algebras. Recall that a {\bf twisted commutative algebra} (tca) over the complex numbers is a commutative unital associative $\bC$-algebra $A$ equipped with a polynomial action of $\GL_{\infty}$; see \cite{expos} for background. The easiest examples of tca's come by taking the symmetric algebra on a polynomial representation of $\GL_{\infty}$: for example, $\Sym(\bC^{\infty})$ or $\Sym(\Sym^2(\bC^{\infty}))$.

TCA's have appeared in several applications in recent years, for instance:
\begin{itemize}
\item Modules over the tca $\Sym(\bC^{\infty})$ are equivalent to $\mathbf{FI}$-modules, as studied in \cite{fimodule}. The structure of the module category was worked out in great detail in \cite{symc1}.
\item Finite length modules over the tca $\Sym(\Sym^2(\bC^{\infty}))$ are equivalent to algebraic representations of the infinite orthogonal group \cite{infrank}.
\item Modules over tca's generated in degree~1 were used to study $\Delta$-modules in \cite{delta}, with applications to syzygies of Segre embeddings.
\end{itemize}

A tca $A$ is {\bf noetherian} if its module category is locally noetherian; explicitly, this means that any submodule of a finitely generated $A$-module is finitely generated. A major open question in the theory, first raised in \cite{delta}, is:

\begin{question}
Is every finitely generated tca noetherian?
\end{question}

So far, our knowledge on this question is extremely limited. For tca's generated in degrees $\le 1$ (or more generally, ``bounded'' tca's), noetherianity was proved in \cite{delta}. (It was later reproved in the special case of $\mathbf{FI}$-modules in \cite{fimodule}.) For the tca's $\Sym(\Sym^2(\bC^{\infty}))$ and $\Sym(\lw^2(\bC^{\infty}))$, noetherianity was proved in \cite{deg2noeth}. No other cases are known. We remark that these known cases of noetherianity, limited though they are, have been crucial in applications.

Since noetherianity is such a difficult property to study, it is useful to consider a weaker notion. A tca $A$ is {\bf topologically noetherian} if every radical ideal is the radical of a finitely generated ideal. The results of \cite{eggermont} show that tca's generated in degrees $\le 2$ are topologically noetherian. Topological noetherianity of the tca $\Sym(\Sym^d(\bC^{\infty}))$ is equivalent to the noetherianity of the space $P_d(\bC^{\infty})$ appearing in Question~\ref{mainques}. Thus Theorem~\ref{mainthm} can be restated as:

\begin{theorem}
The tca $\Sym(\Sym^3(\bC^{\infty}))$ is topologically noetherian.
\end{theorem}

This is the first noetherianity result for an unbounded tca generated in degrees $\ge 3$.

\subsection{A result for tensors}

Using similar methods, we can prove the following result:

\begin{theorem}
The space $P_1(\bk^{\infty}) \,\wh{\otimes}\, P_1(\bk^{\infty}) \,\wh{\otimes}\, P_1(\bk^{\infty})$ is noetherian with respect to the action of the group $\GL_{\infty} \times \GL_{\infty} \times \GL_{\infty}$, where $\wh{\otimes}$ denotes the completed tensor product.
\end{theorem}

We plan to write a short note containing the proof.

\subsection{Outline of paper}

In \S \ref{sec:qrank} we establish a number of basic facts about q-rank. In \S \ref{sec:proof} we use these facts to prove the main theorem. Finally, in \S \ref{sec:example}, we compute the q-rank of the cubic in Example~\ref{ex1}. This example is not used in the proof of the main theorem, but we thought it worthwhile to include one non-trivial computation of our fundamental invariant.

\subsection{Notation and terminology} \label{ss:note}

Throughout we let $\bk$ be an algebraically closed field of characteristic $\ne 2,3$. The symbols $E$, $V$, and $W$ will always denote $\bk$-vector spaces, perhaps infinite dimensional. We write $P_d(V)=\Sym^d(V)^*$ for the space of degree $d$ polynomials on $V$ equipped with the Zariski topology. Precisely, we identify $P_d(V)$ with the spectrum of the ring $\Sym(\Sym^d(V))$. When $V$ is infinite dimensional the elements of $P_d(V)$ are certain infinite series and the functions on $P_d(V)$ are polynomials in coefficients. Whenever we speak of the orbit of an element of $P_d(V)$, we mean its $\GL(V)$ orbit.

\subsection*{Acknowledgements}

We thank Bhargav Bhatt, Jan Draisma, Daniel Erman, Mircea Mustata, and Steven Sam for helpful discussions.

\section{Basic properties of q-rank} \label{sec:qrank}

In this section, we establish a number of basic facts about q-rank. Throughout $V$ will denote a vector space and $f$ a cubic in $P_3(V)$. Initially we allow $V$ to be infinite dimensional, but following Proposition~\ref{prop:inf} it will be finite dimensional (though this is often not necessary).

Our first result is immediate, but worthwhile to write out explicitly.

\begin{proposition}[Subadditivity] \label{prop:subadd}
Suppose $f,g \in P_3(V)$. Then
\begin{displaymath}
\qrank(f+g) \le \qrank(f)+\qrank(g)
\end{displaymath}
\end{proposition}

We defined q-rank from an algebraic point of view (number of terms in a certain sum). We now give a geometric characterization of q-rank that can, at times, be more useful.

\begin{proposition} \label{prop:geom}
We have $\qrank(f) \le r$ if and only if there exists a linear subspace $W$ of $V$ of codimension at most $r$ such that $f \vert_W = 0$.
\end{proposition}

\begin{proof}
First suppose $\qrank(f) \le r$, and write $f=\sum_{i=1}^r \ell_i q_i$. Then we can take $W = \bigcap_{i=1}^r \ker(\ell_i)$. This clearly has the requisite properties.

Now suppose $W$ of codimension $r$ is given. Let $v_{r+1}, v_{r+2}, \ldots$ be a basis for $W$, and complete it to a basis of $V$ be adding vectors $v_1, \ldots v_r$. Let $x_i \in P_1(V)$ be dual to $v_i$. We can then write $f=g+h$, where every term in $g$ uses one of the variables $x_1, \ldots, x_r$, and these variables do not appear in $h$. Since $f \vert_W=0$ by assumption and $g \vert_W=0$ by its definition, we find $h \vert_W=0$. But $h$ only uses the variables $x_{r+1}, x_{r+2}, \ldots$, and these are coordinates on $W$, so we must have $h=0$. Thus every term of $f$ has one of the variables $\{x_1, \ldots, x_r\}$ in it, and so we can write $f=\sum_{i=1}^r x_i q_i$ for appropriate $q_i \in P_2(V)$, which shows $\qrank(f) \le r$.
\end{proof}

\begin{remark}
In the above proposition, $f \vert_W=0$ means that the image of $f$ in $P_3(W)$ is~0. It is equivalent to ask $f(w)=0$ for all $w \in W$.
\end{remark}

The next result shows that one does not lose too much q-rank when passing to subspaces.

\begin{proposition} \label{prop:qsubsp}
Suppose $W \subset V$ has codimension $d$. Then for $f \in P_3(V)$ we have
\begin{displaymath}
\qrank(f)-d \le \qrank(f \vert_W) \le \qrank(f).
\end{displaymath}
\end{proposition}

\begin{proof}
If $f=\sum_{i=1}^r \ell_i q_i$ then we obtain a similar expression for $f \vert_W$, which shows that $\qrank(f \vert_W) \le \qrank(f)$. Suppose now that $\qrank(f \vert_W)=r$, and let $W' \subset W$ be a codimension $r$ subspace such that $f \vert_{W'}=0$ (Proposition~\ref{prop:geom}). Then $W'$ has codimension $r+d$ in $V$, and so $\qrank(f) \le r+d$ (Proposition~\ref{prop:geom} again).
\end{proof}

Our next result shows that if $V$ is infinite dimensional then the q-rank of $f \in P_3(V)$ can be approximated by the q-rank of $f \vert_W$ for a large finite dimensional subspace $W$ of $V$. This will be used at a key juncture to move from an infinite dimensional space down to a finite dimensional one.

\begin{proposition} \label{prop:inf}
Suppose $V=\bigcup_{i \in I} V_i$ (directed union). Then $\qrank(f) = \sup_{i \in I} \qrank(f \vert_{V_i})$.
\end{proposition}

We first give two lemmas. In what follows, for a finite dimensional vector space $W$ we write $\Gr_r(W)$ for the Grassmannian of codimension $r$ subspaces of $W$. For a $\bk$-point $x$ of $\Gr_r(W)$, we write $E_x$ for the corresponding subspace of $W$. By ``variety'' we mean a reduced scheme of finite type over $\bk$.

\begin{lemma} \label{lem:grmap}
Let $W \subset V$ be finite dimensional vector spaces, and let $Z \subset \Gr_r(V)$ be a closed subvariety. Suppose that for every $\bk$-point $z$ of $Z$ the space $E_z \cap W$ has codimension $r$ in $W$. Then there is a unique map of varieties $Z \to \Gr_r(W)$ that on $\bk$-points is given by the formula $E \mapsto E \cap W$.
\end{lemma}

\begin{proof}
Let $\Hom(V, \bk^r)$ be the scheme of all linear maps $V \to \bk^r$, and let $\Surj(V, \bk^r)$ be the open subscheme of surjective linear maps. We identify $\Gr_r(V)$ with the quotient of $\Surj(V, \bk^r)$ by the group $\GL_r$. The quotient map $\Surj(V, \bk^r) \to \Gr_r(V)$ sends a surjection to its kernel. Let $\wt{Z} \subset \Surj(V, \bk^r)$ be the inverse image of $Z$. There is a natural map $\Hom(V, \bk^r) \to \Hom(W, \bk^r)$ given by restricting. By assumption, every closed point of $\wt{Z}$ maps into $\Surj(W, \bk^r)$ under this map. Since $\Surj(W, \bk^r)$ is open, it follows that the map $\wt{Z} \to \Hom(W, \bk^r)$ factors through a unique map of schemes $\wt{Z} \to \Surj(W, \bk^r)$. Since this map is $\GL_r$-equivariant, it descends to the desired map $Z \to \Gr_r(W)$. If $z$ is a $\bk$-point of $Z$ then it lifts to a $\bk$-point $\wt{z}$ of $\wt{Z}$, and the corresponding map $\varphi \colon V \to \bk^r$ has $\ker(\varphi)=E_z$. The image of $z$ in $\Gr_r(W)$ is $\ker(\varphi \vert_W)=E_z \cap W$, which establishes the stated formula for our map.
\end{proof}

\begin{lemma} \label{lem:invsys}
Let $\{Z_i\}_{i \in I}$ be an inverse system of non-empty proper varieties over $\bk$. Then $\varprojlim Z_i(\bk)$ is non-empty.
\end{lemma}

\begin{proof}
If $\bk=\bC$ then $Z_i(\bC)$ is a non-empty compact Hausdorff space, and the result follows from the well-known (and easy) fact that an inverse limit of non-empty compact Hausdorff spaces is non-empty.

For a general field $\bk$, we argure as follows. (We thank Bhargav Bhatt for this argument.) Let $\vert Z_i \vert$ be the Zariski topological space underlying the scheme $Z_i$, and let $Z$ be the inverse limit of the $\vert Z_i \vert$. Since each $\vert Z_i \vert$ is a non-empty spectral space and the transition maps $\vert Z_i \vert \to \vert Z_j \vert$ are spectral (being induced from a map of varieties), $Z$ is also a non-empty spectral space \cite[Lemma~5.24.2, 5.24.5]{stacks}. It therefore has some closed point $z$. Let $z_i$ be the image of $z$ in $\vert Z_i \vert$.

We claim that $z_i$ is closed for all $i$. Suppose not, and let $0 \in I$ be such that $z_0$ is not closed. Passing to a cofinal set in $I$, we may as well assume $0$ is the unique minimal element. Let $\bk(z_i)$ be the residue field of $z_i$, and let $K$ be the direct limit of the $\bk(z_i)$. The point $z_i$ is then the image of a canonical map of schemes $a_i \colon \Spec(K) \to Z_i$. Since $z_0$ is not closed, it admits some specialization, so we may choose a valuation ring $R$ in $K$ and a non-constant map of schemes $b_0 \colon \Spec(R) \to Z_0$ extending $a_0$. Since $Z_i$ is proper, the map $a_i$ extends uniquely to a map $b_i \colon \Spec(R) \to Z_i$. By uniqueness, the $b$'s are compatible with the transition maps, and so we get an induced map $b \colon \vert \Spec(R) \vert \to Z$ extending the map $a \colon \vert \Spec(K) \vert \to Z$. Since $\vert b_0 \vert$ is induced from $b$, it follows that $b$ is non-constant. The image of the closed point in $\Spec(R)$ under $b$ is then a specialization of $z$, contradicting the fact that $z$ is closed. This completes the claim that $z_i$ is closed.

Since $z_i$ is closed, it is the image of a unique map $\Spec(\bk) \to Z_i$ of $\bk$-schemes. By uniqueness, these maps are compatible, and so give an element of $\varprojlim Z_i(\bk)$.
\end{proof}

\begin{proof}[Proof of Proposition~\ref{prop:inf}]
First suppose that $V_i$ is finite dimensional for all $i$. For $i \le j$ we have $\qrank(f \vert_{V_i}) \le \qrank(f \vert_{V_j})$ by Proposition~\ref{prop:qsubsp}, and so either $\qrank(f \vert_{V_i}) \to \infty$ or $\qrank(f \vert_{V_i})$ stabilizes. If $\qrank(f \vert_{V_i}) \to \infty$ then $\qrank(f)=\infty$ by Proposition~\ref{prop:qsubsp} and we are done. Thus suppose $\qrank(f \vert_{V_i})$ stabilizes. Replacing $I$ with a cofinal subset, we may as well assume $\qrank(f \vert_{V_i})$ is constant, say equal to $r$, for all $i$. We must show $\qrank(f)=r$. Proposition~\ref{prop:qsubsp} shows that $r \le \qrank(f)$, so it suffices to show $\qrank(f) \le r$.

Let $Z_i \subset \Gr_r(V_i)$ be the closed subvariety consisting of all codimension $r$ subspaces $E \subset V_i$ such that $f \vert_E=0$. This is non-empty by Proposition~\ref{prop:geom} since $f \vert_{V_i}$ has q-rank $r$. Suppose $i \le j$ and $z$ is a $\bk$-point of $Z_j$, that is, $E_z$ is a codimension $r$ subspace of $V_j$ on which $f$ vanishes. Of course, $f$ then vanishes on $V_i \cap E_z$, which has codimension at most $r$ in $V_i$. Since $f \vert_{V_i}$ has q-rank exactly $r$, it cannot vanish on a subspace of codimension less than $r$ (Proposition~\ref{prop:geom}), and so $V_i \cap E_z$ must have codimension exactly $r$. Thus by Lemma~\ref{lem:grmap}, intersecting with $V_i$ defines a map of varieties $Z_j \to \Gr_r(V_i)$. This maps into $Z_i$, and so for $i \le j$ we have a map $Z_j \to Z_i$. These maps clearly define an inverse system.

Appealing to Lemma~\ref{lem:invsys} we see that $\varprojlim Z_i(\bk)$ is non-empty. Let $\{z_i\}_{i \in I}$ be a point in this inverse limit, and put $E_i=E_{z_i}$. Thus $E_i$ is a codimension $r$ subspace of $V_i$ on which $f$ vanishes, and for $i \le j$ we have $E_j \cap V_i=E_i$. It follows that $E=\bigcup_{i \in I} E_i$ is a codimension $r$ subspace of $V$ on which $f$ vanishes, which shows $\qrank(f) \le r$ (Proposition~\ref{prop:geom}).

We now treat the general case, where the $V_i$ may not be finite dimensional. Write $V_i=\bigcup_{j \in J_i} W_j$ with $W_j$ finite dimensional. Then $V=\bigcup_{i \in I} \bigcup_{j \in J_i} W_j$, so
\begin{displaymath}
\qrank(f) = \sup_{i \in I} \sup_{j \in J_i} \qrank(f \vert_{W_j}) = \sup_{i \in I} \qrank(f \vert_{V_i}).
\end{displaymath}
This completes the proof.
\end{proof}

For the remainder of this section we assume that $V$ is finite dimensional. If $V$ is $d$-dimensional then the q-rank of any cubic in $P_3(V)$ is obviously bounded above by $d$. The next result gives an improved bound, and will be crucial in what follows.

\begin{proposition} \label{prop:qbd}
Suppose $\dim(V)=d$. Then $\qrank(f) \le d-\xi(d)$, where
\begin{displaymath}
\xi(d) = \left\lfloor \frac{\sqrt{8d+17}-3}{2} \right\rfloor.
\end{displaymath}
Note that $\xi(d) \approx \sqrt{2d}$.
\end{proposition}

\begin{proof}
Let $k$ be the largest integer such that $\binom{k+1}{2}+k-1 \le d$. Then the hypersurface $f=0$ contains a linear subspace of dimension at least $k$ by \cite[Lemma~3.9]{hmp}. It follows from Proposition~\ref{prop:geom} that $\qrank(f) \le d-k$. Some simple algebra shows that $k=\xi(d)$.
\end{proof}

Suppose that $f=\sum_{i=1}^n \ell_i q_i$ is a cubic. Eventually, we want to show that if $f$ has large q-rank then its orbit under $\GL(V)$ is large. For studying the orbit, it would be convenient if the $\ell_i$'s and the $q_i$'s were in separate sets of variables, as then they could be moved independently under the group. This motivates the following definition:

\begin{definition}
We say that a cubic $f \in P_3(V)$ is {\bf separable}\footnote{This notion of separable is unrelated to the notion of separability of univariate polynomials. We do not expect this to cause confusion.} if there is a direct sum decomposition $V=V_1 \oplus V_2$ and an expression $f=\sum_{i=1}^n \ell_i q_i$ with $\ell_i \in P_1(V_1)$ and $q_i \in P_2(V_2)$.
\end{definition}

Now, if we have a cubic $f$ of high q-rank we cannot conclude, simply based on its high q-rank, that it is separable. Fortunately, the following result shows that if we are willing to degenerate $f$ a bit (which is fine for our ultimate applications), then we can make it separable, while retaining high q-rank.

\begin{proposition} \label{prop:srk}
Suppose that $f \in P_3(V)$ has q-rank $r$. Then the orbit-closure of $f$ contains a separable cubic $g$ satisfying $\tfrac{1}{2} \xi(r) \le \qrank(g)$.
\end{proposition}

\begin{proof}
Let $\{x_i\}$ be a basis for $P_1(V)$. After possibly making a linear change of variables, we can write $f = \sum_{i=1}^r x_i q_i$. Write $f=f_1+f_2+f_3$, where $f_i$ is homogeneous of degree $i$ in the variables $\{x_1,\ldots,x_r\}$. Since $f_3$ has degree~3 in the variables $\{x_1,\ldots,x_r\}$, it can contain no other variables, and can thus be regarded as an element of $P_3(\bk^r)$. Therefore, by Proposition~\ref{prop:qbd}, we have $\qrank(f_3) \le r-\xi(r)$. After possibly making a linear change of variables in $\{x_1, \ldots, x_r\}$, we can write $f_3=\sum_{i=\xi(r)+1}^r x_i q_i'$ for some $q_i'$. Let $f'$ (resp.\ $f'_j$) be the result of setting $x_i=0$ in $f$ (resp.\ $f_j$), for $\xi(r)<i \le r$. We have $\qrank(f') \ge \xi(r)$ by Proposition~\ref{prop:qsubsp}. Of course, $f'_3=0$, so $f'=f'_1+f'_2$. By subadditivity (Proposition~\ref{prop:subadd}), at least one of $f'_1$ or $f'_2$ has q-rank $\ge \tfrac{1}{2} \xi(r)$.

We have $f_1=\sum_{i=1}^r x_i q_i''$ where $q_i''$ is a quadratic form in the variables $x_i$ with $i>r$. Thus $f_1$, and $f_1'$, is separable. We have $f_2=\sum_{1 \le i \le j \le r} x_i x_j \ell_{i,j}$ where $\ell_{i,j}$ is a linear form in the variables $x_i$ with $i>r$. Thus $f_2$, and $f_2'$, is separable.

To complete the proof, it suffices to show that $f'_1$ and $f'_2$ belong to the orbit-closure of $f$, as we can then take $g=f'_1$ or $g=f'_2$. It is clear that $f'$ is in the orbit-closure of $f$, so it suffices to show that $f'_1$ and $f'_2$ are in the orbit-closure of $f'$. Consider the element $\gamma_t$ of $\GL_n$ defined by
\begin{displaymath}
\gamma_t(x_i) = \begin{cases}
t^2 x_i & 1 \le i \le r \\
t^{-1} x_i & r < i \le n \end{cases}
\end{displaymath}
Then $\gamma_t(f'_1)=f'_1$ and $\gamma_t(f'_2)=t^3 f'_2$. Thus $\lim_{t \to 0} \gamma_t(f') = f'_1$. A similar construction shows that $f'_2$ is in the orbit-closure of $f'$.
\end{proof}

Suppose that $f=\sum_{i=1}^n \ell_i q_i$ is a cubic of high q-rank. One would like to be able to conclude that the $q_i$ then have high ranks as well. We now prove two results along this line. For a linear subspace $Q \subset P_2(V)$, we let $\maxrank(Q)$ be the maximum of the ranks of elements of $Q$, and we let $\minrank(Q)$ be the minimum of the ranks of the non-zero elements of $Q$ (or 0 if $Q=0$).

\begin{proposition} \label{prop:maxrank}
Suppose $f=\sum_{i=1}^n \ell_i q_i$ has q-rank $r$, and let $Q \subset P_2(V)$ be the span of the $q_i$. Then for every subspace $Q'$ of $Q$ we have
\begin{displaymath}
\codim(Q:Q')+\maxrank(Q') \ge r.
\end{displaymath}
\end{proposition}

\begin{proof}
We may as well assume that $\ell_i$ and $q_i$ are linearly independent. Thus $\dim(Q)=n$. Let $Q'$ be a subspace of dimension $n-d$. After making a linear change of variables in the $q$'s and $\ell$'s, we may as well assume that $Q'$ is the span of $q_1, \ldots, q_{n-d}$. Let $t=\maxrank(Q')$. We must show that $d+t \ge r$. Let $q' \in Q'$ have rank $t$. Choose a  basis $\{x_i\}$ of $P_1(V)$ so that $q'=x_1^2+\cdots+x_t^2$. If some $q_i$ for $1 \le i \le n-d$ had a term of the form $x_j x_k$ with $j,k>t$ then some linear combination of $q_i$ and $q'$ would have rank $>t$, a contradiction. Thus every term of $q_i$, for $1 \le i \le n-d$, has a variable of index $\le t$, and so we can write $q_i=\sum_{j=1}^t x_j m_{i,j}$ where $m_{i,j} \in P_1(V)$. But now
\begin{displaymath}
f=\sum_{i=1}^{n-d} \ell_i q_i + \sum_{i=n-d+1}^n \ell_i q_i
=\sum_{j=1}^t x_j q_j' + \sum_{i=n-d+1}^n \ell_i q_i
\end{displaymath}
where $q_j'=\sum_{i=1}^{n-d} \ell_i m_{i,j}$. This shows $r=\qrank(f) \le t+d$, which completes the proof.
\end{proof}

In our eventual application, it is actually $\minrank$ that is more important than $\maxrank$. Fortuantely, the above result on $\maxrank$ automatically gives a result for $\minrank$, thanks to the following general proposition.

\begin{proposition} \label{prop:minrank}
Let $Q \subset P_2(V)$ be a linear subspace and let $r$ be a positive integer. Suppose that
\begin{displaymath}
\codim(Q:Q')+\maxrank(Q') \ge r
\end{displaymath}
holds for all linear subspaces $Q' \subset Q$. Let $k$ and $s$ be positive integers satisfying
\begin{equation}
\label{eq:kscond}
(2^k-1)(s-1)+k \le r.
\end{equation}
Then there exists a $k$-dimensional linear subspace $Q' \subset Q$ with $\minrank(Q') \ge s$.
\end{proposition}

\begin{lemma} \label{lem:minrank}
Let $q_1, \ldots, q_n \in P_2(V)$ be quadratic forms of rank $<s$. Suppose there is a linear combination of the $q$'s that has rank at least $t$. Then there is a linear combination $q'$ of the $q$'s satisfying $t \le \rank(q') \le t+s-2$.
\end{lemma}

\begin{proof}
Let $q'=\sum_{i=1}^k a_i q_i$ be a linear combination of the $q$'s with $\rank \ge t$ and $k$ minimal. Since $\rank(q_k) \le s-1$, it follows that $\rank(q'-a_k q_k) \ge \rank(q')-(s-1)$. Thus if $\rank(q') \ge t+s-1$ then $\sum_{i=1}^{k-1} a_i q_i$ would have rank $\ge t$, contradicting the minimality of $k$. Therefore $\rank(q') \le t+s-2$.
\end{proof}

\begin{proof}[Proof of Proposition~\ref{prop:minrank}]
Let $q_1, \ldots, q_n$ be a basis for $Q$ so that $(\rank(q_1), \ldots, \rank(q_n))$ is lexicographically minimal. In particular, this implies that $\rank(q_1) \le \cdots \le \rank(q_n)$. If $\rank(q_{n-k+1}) \ge s$ then lexicographic minimality ensures that any non-trivial linear combination of $q_{n-k+1}, \ldots, q_n$ has rank at least $s$, and so we can take $Q'$ to be the span of these forms. Thus suppose that $\rank(q_{n-k+1})<s$. In what follows, we put $m_i=(2^i-1)(s-1)+1$. Note that $m_k \le r$. In fact, $n-r+m_k \le n-k+1$, and so $\rank(q_{n-r+m_k})<s$.

For $1 \le \ell \le k$, consider the following statement:
\begin{itemize}
\item[$(S_{\ell})$] There exist linearly independent $p_1, \ldots, p_{\ell}$ such that: (i) $p_i$ is a linear combination of $q_1, \ldots, q_{n-r+m_i}$; (ii) $m_i \le \rank(p_i) \le m_i+s-2$; and (iii) the span of $p_1, \ldots, p_{\ell}$ has minrank at least $s$.
\end{itemize}
We will prove $(S_{\ell})$ by induction on $\ell$. Of course, $(S_k)$ implies the proposition.

First consider the case $\ell=1$. The statement $(S_1)$ asserts that there exists a non-zero linear combination $p$ of $q_1, \ldots, q_{n-r+s}$ such that $s \le \rank(p) \le 2s-2$. Since the span of $q_1, \ldots, q_{n-r+s}$ has codimension $r-s$ in $Q$, our assumption guarantees that some linear combination $p$ of these forms has rank at least $s$. Since each form has rank $<s$, Lemma~\ref{lem:minrank} ensures we can find $p$ with $\rank(p) \le s+(s-2)$.

We now prove $(S_{\ell})$ assuming $(S_{\ell-1})$. Let $(p_1, \ldots, p_{\ell-1})$ be the tuple given by $(S_{\ell-1})$. The span of $q_1, \ldots, q_{n-r+m_{\ell}}$ has codimension $r-m_{\ell}$ in $Q$, and so our assumption guarantees that some linear combination $p_{\ell}$ has rank at least $m_{\ell}$. By Lemma~\ref{lem:minrank}, we can ensure that this $p_{\ell}$ has rank at most $m_{\ell}+s-2$. Thus (i) and (ii) in $(S_{\ell})$ are established.

We now show that any non-trivial linear combination $\sum_{i=1}^{\ell} \lambda_i p_i$ has rank at least $s$, which will show that the $p$'s are linearly independent and establish (iii) in $(S_{\ell})$. If $\lambda_{\ell}=0$ then the rank is at least $s$ by the assumption on $(p_1, \ldots, p_{\ell-1})$. Thus assume $\lambda_{\ell} \ne 0$. We have
\begin{displaymath}
\rank \left( \sum_{i=1}^{\ell-1} \lambda_i p_i \right) \le \sum_{i=1}^{\ell-1} \rank(p_i) \le
\sum_{i=1}^{\ell-1} (m_i+s-2) = m_{\ell}-s.
\end{displaymath}
Since $\rank(p_{\ell}) \ge m_{\ell}$, we thus see that $\sum_{i=1}^{\ell} \lambda_i p_i$ has rank at least $s$, which completes the proof.
\end{proof}

\begin{remark}
Proposition~\ref{prop:minrank} is not specific to ranks of quadratic forms: it applies to any subadditive invariant on a vector space.
\end{remark}

Combining the Propositions~\ref{prop:maxrank} and~\ref{prop:minrank}, we obtain:

\begin{corollary} \label{cor:minrank}
Suppose $f=\sum_{i=1}^n \ell_i q_i$ has q-rank $r$, let $Q$ be the span of the $q_i$'s, and let $k$ and $s$ be positive integers such that \eqref{eq:kscond} holds. Then there exists a $k$-dimensional linear subspace $Q' \subset Q$ with $\minrank(Q') \ge s$.
\end{corollary}

\section{Proof of Theorem~\ref{mainthm}} \label{sec:proof}

We now prove the main theorems of the paper. We require the following result (see \cite[Proposition~3.3]{eggermont} and its proof):

\begin{theorem} \label{thm:deg2}
Let $x$ be a point in $P_2(V)^n \times P_1(V)^m$, with $V$ finite dimensional. Write $x$ as $(q_1, \ldots, q_n; \ell_1, \ldots, \ell_m)$, and let $Q \subset P_2(V)$ be the span of the $q_i$. Let $W$ be a $d$-dimensional subspace of $V$. Suppose that $\ell_1, \ldots, \ell_m$ are linearly independent and that $\minrank(Q) \ge dn2^n+2(n+1)m$. Then the orbit-closure of $x$ surjects onto $P_2(W)^n \times P_1(W)^m$.
\end{theorem}

We begin by proving an analog of the above theorem for $P_3(V)$:

\begin{theorem} \label{thm:surj}
Suppose $V$ is finite dimensional. Let $f \in P_3(V)$ have q-rank $r$ and let $W$ be a $d$-dimensional subspace of $V$ with
\begin{displaymath}
(2^d-1)(d^2 2^d+2(d+1)d-1)+d \le \tfrac{1}{2} \xi(r).
\end{displaymath}
Then the orbit-closure of $f$ surjects onto $P_3(W)$.
\end{theorem}

\begin{proof}
Applying Proposition~\ref{prop:srk}, let $g$ be a separable cubic in the orbit-closure of $f$ satisfying $\tfrac{1}{2} \xi(r) \le \qrank(g)$. Write $g=\sum_{i=1}^n \ell_i q_i$ where $\ell_i \in P_1(V_1)$ and $q_i \in P_2(V_2)$ and $V=V_1 \oplus V_2$ and the $\ell$'s and $q$'s are linearly independent. Let $Q$ be the span of the $q$'s. Put $s=d^2 2^d+2(d+1) d$ and $k=d$. Note that
\begin{displaymath}
(2^k-1)(s-1)+k \le \tfrac{1}{2} \xi(r).
\end{displaymath}
By Corollary~\ref{cor:minrank} we can therefore find a $k=d$ dimensional subspace $Q'$ of $Q$ with $\minrank(Q') \ge s$. Making a linear change of variables, we can assume $Q'$ is the span of $q_1, \ldots, q_d$. Let $g'=\sum_{i=1}^d \ell_i q_d$. This is in the orbit-closure of $g$ (and thus $f$) since it is obtained by setting $\ell_i=0$ for $i>d$. It is crucial here that the $q$'s and $\ell$'s are in different sets of variables, so that setting some $\ell$'s to 0 does not change the $q$'s. By Theorem~\ref{thm:deg2}, the orbit closure of $(q_1, \ldots, q_d, \ell_1, \ldots, \ell_d)$ in $P_2(V)^d \times P_1(V)^d$ surjects onto $P_2(W)^d \times P_1(W)^d$. Now let $h \in P_3(W)$. Since $\dim(W)=d$ we can write $h=\sum_{i=1}^d \ell_i' q_i'$ with $\ell_i' \in P_1(W)$ and $q_i' \in P_2(W)$. Pick $\gamma_t \in \GL(V)$ such that $(q_1', \ldots, q_d'; \ell_1', \ldots, \ell_d')$ in the image of $\lim_{t \to 0} \gamma_t \cdot (q_1, \ldots, q_d; \ell_1 \ldots, \ell_d)$. Then $h$ is the image of $\lim_{t \to 0} \gamma_t \cdot g'$, which completes the proof.
\end{proof}

\begin{corollary}[Theorem~\ref{mainthm3}]
Suppose that $f \in P_3(V)$ has q-rank $r>\exp(240)$ and let $W$ be a subspace of $V$ of dimension $d$ with $d<\tfrac{1}{3} \log{r}$ Then the orbit-closure of $f$ surjects onto $P_3(W)$.
\end{corollary}

\begin{proof}
By definition of $\xi$, we have $a \le \xi(r)$ (for an integer $a$) if and only if $\binom{a+1}{2}+a-1 \le r$. Thus the condition in the theorem is equivalent to $\binom{D+1}{2}+D-1 \le r$, where $D$ is the left side of the inequality in the theorem. This expression is equal to $d^4 16^d$ plus lower order terms, and is therefore less than $20^d$ for $d \gg 0$; in fact, $d \ge 80$ is sufficient. Thus for $d \ge 80$ it is enough that $d < \frac{\log{r}}{\log{20}}$; since $\log(20)<3$, it is enough that $d<\tfrac{1}{3} \log(r)$. Thus for $80 \le d \le \tfrac{1}{3} \log(r)$, the orbit closure of $f$ surjects onto $P_3(W)$. But it obviously then surjects onto smaller subspaces as well, so we only need to assume $80 \le \tfrac{1}{3} \log(r)$.
\end{proof}

\begin{theorem}[Theorem~\ref{mainthm2}]
Let $V$ be infinite dimensional. Suppose $Z \subset P_3(V)$ is Zariski closed, $\GL(V)$-stable, and contains elements of arbitrarily high q-rank. Then $Z=P_3(V)$.
\end{theorem}

\begin{proof}
It suffices to show that $Z$ surjects onto $P_3(W)$ for all finite dimensional $W \subset V$. Thus let $W$ of dimension $d$ be given. Let $r$ be sufficiently large so that the inequality in Theorem~\ref{thm:surj} is satisfied and let $f \in Z$ have q-rank at least $r$. By Proposition~\ref{prop:inf}, there exists a finite dimensional subspace $V'$ of $V$ containing $W$ such that $f \vert_{V'}$ has q-rank at least $r$. Theorem~\ref{thm:surj} implies that the orbit-closure of $f \vert_{V'}$ surjects onto $P_3(W)$. Since $Z$ surjects onto the orbit closure of $f \vert_{V'}$, the result follows.
\end{proof}

It was explained in the introduction how this implies Theorem~\ref{mainthm}, so the proof is now complete.

\section{A computation of q-rank} \label{sec:example}

Fix a positive integer $n$, and consider the cubic
\begin{displaymath}
f = x_1 y_1 z_1 + \cdots + x_n y_n z_n
\end{displaymath}
in the polynomial ring $\bk[x_i, y_i, z_i]_{1 \le i \le n}$ introduced in Example~\ref{ex1}. We now show:

\begin{proposition}
The above cubic $f$ has q-rank $n$.
\end{proposition}

It is clear that $\qrank(f) \le n$. To prove equality, it suffices by Proposition~\ref{prop:geom} to show that $f \vert_V \ne 0$ if $V$ is a codimension $n-1$ subspace of $\bk^{3n}$. This is exactly the content of the following proposition:

\begin{proposition}
Let $V$ be a vector space of dimension $2n+1$ and let $(x_i,y_i,z_i)_{1 \le i \le n}$ be a collection of elements that span $P_1(V)$. Then $f=x_1y_1z_1+\cdots+x_ny_nz_n \in P_3(V)$ is non-zero.
\end{proposition}

\begin{proof}
Arrange the given elements in a matrix as follows:
\begin{displaymath}
\begin{pmatrix}
x_1 & y_1 & z_1 \\
\vdots & \vdots & \vdots \\
x_n & y_n & z_n
\end{pmatrix}
\end{displaymath}
Note that we are free to permute the rows and apply permutations within a row without changing the value of $f$, e.g., we can switch the values of $x_1$ and $y_1$, or switch $(x_1, y_1, z_1)$ with $(x_2, y_2, z_2)$, without changing $f$. We now proceed to find a basis for $V$ among the elements in the matrix according to the following three-phase procedure.

{\it Phase 1.} Find a non-zero element of the matrix, and move it (using the permutations mentioned above) to the $x_1$ position. Now in rows $2, \ldots, n$ find an element that is not in the span of $x_1$ (if one exists) and move it to the $x_2$ position. Now in rows $3, \ldots, n$ find an element that is not in the span of $x_1$ and $x_2$ (if one exists) and move it to the $x_3$ position. Continue in this manner until it is no longer possible; suppose we go $r$ steps. At this point, $x_1, \ldots, x_r$ are linearly independent, and $x_i$, $y_i$, and $z_i$, for $r<i$ all belong to their span.

{\it Phase 2.} From rows $1, \ldots r$ find an element in the second or third column not in the span of $x_1, \ldots, x_r$ and move it (using permutations that fix the first column) to the $y_1$ position. Next from rows $2, \ldots, r$ find an element in the second or third column not in the span of $x_1, \ldots, x_r, y_1$ and move it to the $y_2$ position. Continue in this manner until it is no longer possible; suppose we go $s$ steps. At this point, $x_1, \ldots, x_r, y_1, \ldots, y_s$ form a linearly independent set, and the elements $y_i, z_i$ for $s < i \le r$ belong to their span. The conclusion from Phase~1 still holds as well.

{\it Phase 3.} Now carry out the same procedure in the third column. That is, from rows $1, \ldots, s$ find an element in the third column not in the span of $x_1, \ldots, x_r, y_1, \ldots, y_s$ and move it (by permuting rows) to the $z_1$ position. Then from rows $2, \ldots, s$ find an element in the third column not in the span of $x_1, \ldots, x_r, y_1, \ldots, y_s, z_1$ and move it to the $z_2$ position. Continue in this manner until it is no longer possible; suppose we go $t$ steps. At this point, $x_1, \ldots, x_r, y_1, \ldots, y_s, z_1, \ldots, z_t$ forms a basis of $V$. The conclusions from Phases~1 and~2 still hold.

For clarity, we write $X_1, \ldots, X_r, Y_1, \ldots, Y_s, Z_1, \ldots, Z_t$ for our basis. We note that because $\dim(V)>2n$ we must have $t \ge 1$. The ring $\Sym(V^*)$ is identified with the polynomial ring in the $X$, $Y$, $Z$ variables. We now determine the coefficient of $X_1 Y_1 Z_1$ in $m_i=x_i y_i z_i$. If $i>r$ then $m_i$ has degree~3 in the $X$ variables, and so the coefficient is~0. If $s<i \le r$ then $m_i$ has degree~0 in the $Z$ variables, and so again the coefficient is~0. Finally, suppose that $i<s$. Then $m_i=X_iY_i z_i$. The only way this can contain $X_1Y_1Z_1$ is if $i=1$. We thus see that the coefficient of $X_1Y_1Z_1$ in $m_i$ is~0 except for $i=1$, in which case it is~1, and so $f=\sum_{i=1}^n m_i$ is non-zero.
\end{proof}

\begin{remark}
It follows from the above results and Proposition~\ref{prop:inf} that the cubic $\sum_{i=1}^{\infty} x_i y_i z_i$ has infinite q-rank.
\end{remark}

\end{document}